\documentclass[10pt,reqno]{amsart}

\usepackage[a4paper,headheight=0.3cm]{geometry}
\usepackage{fancyhdr}
\usepackage{url}
\usepackage[pdftitle={Decision making under uncertainty using imprecise probabilities},%
pdfauthor={Matthias C. M. Troffaes},%
pdfkeywords={decision, optimality, uncertainty, probability, maximality, E-admissibility, maximin, lower prevision}]{hyperref}

\usepackage{pb-diagram}
\usepackage{amsmath}
\usepackage{amssymb}
\usepackage{enumerate}
\usepackage{stmaryrd}

\newcommand{\SetR}{{\mathbb{R}}}
\newcommand{\SetN}{{\mathbb{N}}}
\newcommand{\optop}[1]{{\mathrm{opt}}_{#1}}
\newcommand{\opts}[2]{\optop{#1}\left(#2\right)}
\newcommand{\UpSet}[3]{{}{\uparrow_{#1}^{#2}\!\!#3}}

\newcommand{\IsALinPrev}{\mathbf{E}}
\newcommand{\IsACohLowPrev}{\underline{\mathbf{E}}}
\newcommand{\IsACohUppPrev}{\overline{\mathbf{E}}}

\newcommand{\Prev}[1]{\mathit{#1}}
\newcommand{\LowPrev}[1]{\protect\underline{\Prev{#1}}}
\newcommand{\UppPrev}[1]{\protect\overline{\Prev{#1}}}

\newcommand{\LinNatXt}[1]{{\IsALinPrev}_{#1}}
\newcommand{\LowNatXt}[1]{{\IsACohLowPrev}_{#1}}
\newcommand{\UppNatXt}[1]{{\IsACohUppPrev}_{#1}}

\newcommand{\sggen}{>} 
\newcommand{\Psg}[1]{\sggen_{#1}} 
\newcommand{\Pintvalsg}[1]{\sqsupset_{#1}}
\newcommand{\maxs}[2]{\mathrm{max}_{#1}\left(#2\right)}

\newcommand{\ud}{\mathrm{d}}

\newcommand{\gambles}{\mathcal{L}(\mathcal{X})}

\theoremstyle{plain}
\newtheorem{thm}{Theorem}
\newtheorem{lem}[thm]{Lemma}
\theoremstyle{definition}
\newtheorem{defn}[thm]{Definition}

\begin{document}

\title{Decision Making under Uncertainty using Imprecise Probabilities}
\author{Matthias C. M. Troffaes}
\address{Carnegie Mellon University, Department of Philosophy, Baker Hall 135, Pittsburgh, PA 15213, US}
\email{matthias.troffaes@gmail.com}

\begin{abstract}
  Various ways for decision making with imprecise
  probabilities---admissibility, maximal expected utility, maximality,
  E-admissibility, $\Gamma$-maximax, $\Gamma$-maximin, all of which
  are well-known from the literature---are discussed and compared. We
  generalize a well-known sufficient condition for existence of
  optimal decisions. A simple numerical example shows how these
  criteria can work in practice, and demonstrates their differences.
  Finally, we suggest an efficient approach to calculate optimal
  decisions under these decision criteria.
\end{abstract}


\maketitle
\thispagestyle{fancy}

\section{Introduction}

Often, we find ourselves in a situation where we have to make some
decision $d$, which we may freely choose from a set $D$ of available
decisions. Usually, we do not choose $d$ arbitrarily in $D$: indeed,
we wish to make a decision that performs best according to some
criterion, i.e., an \emph{optimal} decision. It is commonly
assumed that each decision $d$ induces a real-valued gain $J_d$: in
that case, a decision $d$ is considered optimal in $D$ if it induces
the highest gain among all decisions in $D$. This holds for instance
if each decision induces a lottery over some set of rewards, and these
lotteries form an ordered set satisfying the axioms of von Neumann
Morgenstern \cite{1944:neumann}, or more generally, the axioms of for
instance Herstein and Milnor \cite{1953:herstein:milnor}, if we wish
to account for unbounded gain.

So, we wish to identify the set $\opts{}{D}$ of all decisions that
induce the highest gain. Since, at this stage, there is no uncertainty regarding the
gains $J_d$, $d\in D$, the solution is simply
\begin{equation}
  \label{eq:optimal:simple}
  \opts{}{D}=\arg\max_{d\in D}J_d.
\end{equation}
Of course, $\opts{}{D}$ may be empty; however, if the
set $\{J_d\colon d\in D\}$ is a compact subset of $\SetR$---this
holds for instance if $D$ is finite---then $\opts{}{D}$ contains at
least one element. Secondly, note that even if $\opts{}{D}$ contains
more than one decision, all decisions $d$ in $\opts{}{D}$ induce the same
gain $J_d$; so, if, in the end, the gain is all that matters, it
suffices to identify only one decision $d^*$ in $\opts{}{D}$---often,
this greatly simplifies the analysis.

However, in many situations, the gains $J_d$ induced by decisions $d$
in $D$ are influenced by variables whose values are uncertain.
Assuming that these variables can be modelled through a random
variable $X$ that takes values in some set $\mathcal{X}$ (the
\emph{possibility space}), it is customary to consider the gain $J_d$
as a so-called \emph{gamble} on $X$, that is, we view $J_d$ as a
real-valued gain that is a bounded function of $X$, and that is expressed in a
fixed state-independent utility scale. So, $J_d$ is a bounded
$\mathcal{X}$--$\SetR$-mapping, interpreted as an uncertain gain:
taking decision $d$, we receive an amount of utility $J_d(x)$ when $x$
turns out to be the realisation of $X$. For the sake of simplicity, we
shall assume that the outcome $x$ of $X$ is independent of the
decision $d$ we take: this is called \emph{act-state independence}.
What decision should we take?

Irrespective of our beliefs about $X$, a decision $d$ in $D$ is not
optimal if its gain gamble $J_d$ is \emph{point-wise dominated} by a
gain gamble $J_e$ for some $e$ in $D$, i.e., if there is an $e$
in $D$ such that $J_e(x)\ge J_d(x)$ for all $x\in\mathcal{X}$ and
$J_e(x)>J_d(x)$ for at least one $x\in\mathcal{X}$: choosing $e$
guarantees a higher gain than choosing $d$, possibly strictly higher,
regardless of the realisation of $X$. So, as a first selection, let us
remove all decisions from $D$ whose gain gambles are point-wise
dominated (see Berger \cite[Section~1.3.2, Definition~5~ff.,
p.~10]{1985:berger}):
\begin{equation}
  \label{eq:optimal:pointwise}
  \opts{\ge}{D}:=
  \{d\in D\colon (\forall e\in D)(J_e\not\ge J_d\text{ or }J_e=J_d)\}
\end{equation}
where $J_e\ge J_d$ is understood to be
point-wise, and $J_e\not\ge J_d$ is understood to be the negation of
$J_e\ge J_d$.
The decisions in $\opts{\ge}{D}$ are called \emph{admissible}, the
other decisions in $D$ are called \emph{inadmissible}. Note that we
already recover Eq.~\eqref{eq:optimal:simple} if there is no
uncertainty regarding the gains $J_d$, i.e., if all $J_d$ are
constant functions of $X$. When do admissible decisions exist? The set
$\opts{\ge}{D}$ is non-empty if $\{J_d\colon d\in D\}$ is a non-empty
and weakly compact subset of the set $\gambles$ of all gambles on
$\mathcal{X}$ (see Theorem~\ref{thm:existence:admissible} further on).
Note that this condition is sufficient, but not necessary.

In what follows, we shall try to answer the following question: given
additional information about $X$, how can we further reduce the set
$\opts{\ge}{D}$ of admissible decisions? The paper is structured as
follows. Section~\ref{sec:meu:questionmark} discusses the classical
approach of maximising expected utility, and explains why it is not
always a desirable criterion for selecting optimal decisions. Those
problems are addressed in Section~\ref{sec:generalising:to:improb},
discussing alternative approaches to deal with uncertainty and
optimality, all of which attempt to overcome the issues raised in
Section~\ref{sec:meu:questionmark}, and all of which are known from
the literature. Finally, Section~\ref{sec:the:right:one} compares
these alternative approaches, and explains how optimal decisions can
be obtained in a computationally efficient way. A few technical
results are deferred to the appendix, where we, among other things,
generalize a well-known technical condition on the existence of
optimal decisions.

\section{Maximising Expected Utility?}
\label{sec:meu:questionmark}

In practice, beliefs about $X$ are often modelled by a (possibly
finitely additive) probability measure $\mu$ on a field $\mathcal{F}$
of subsets of $\mathcal{X}$, and one then arrives at a set of optimal
decisions by maximising their expected utility with respect to $\mu$;
see for instance Raiffa and Schlaifer \cite[Section~1.1.4,
p.~6]{1961:raiffa:schlaifer}, Levi \cite[Section~4.8, p.~96,
ll.~23--26]{1983:levi}, or Berger \cite[Section~1.5.2, Paragraph~I,
p.~17]{1985:berger}. Assuming that the field $\mathcal{F}$ is
sufficiently large such that the gains $J_d$ are measurable with
respect to $\mathcal{F}$---this means that every $J_d$ is a uniform limit of
$\mathcal{F}$-simple gambles---the \emph{expected utility} of the gain
gambles $J_d$ is given by:
\begin{equation*}
  \LinNatXt{\mu}(J_d):=\int J_d\ud\mu,
\end{equation*}
where we take for instance the Dunford integral on the right hand
side; see Dunford \cite[p.~443, Sect.~3]{1935:dunford}, and Dunford
and Schwartz \cite[Part~I, Chapter~III, Definition~2.17,
p.~112]{1957:dunford:schwartz}---this linear
integral extends the usual textbook integral (see for instance Kallenberg \cite[Chapter~1]{2002:kallenberg}) to case where $\mu$ is not 
$\sigma$-additive.
Recall that we have assumed act-state independence: $\mu$ is
independent of $d$.

As far as it makes sense to rank decisions
according to the expected utility of their gain gambles, we should
\emph{maximise expected utility}:
\begin{equation}
  \label{eq:optimal:meu}
  \opts{\LinNatXt{\mu}}{D}:=\arg\max_{d\in\opts{\ge}{D}}\LinNatXt{\mu}(J_d).
\end{equation}

When do optimal solutions exist?  The set $\opts{\LinNatXt{\mu}}{D}$
is guaranteed to be non-empty if $\{J_d\colon d\in D\}$ is a non-empty
and compact subset of the set $\gambles$ of all gambles on
$\mathcal{X}$, with respect to the supremum norm. Actually, this
technical condition is sufficient for existence with regard to all of
the optimality conditions we shall discuss further on. Therefore,
without further ado, we shall assume that $\{J_d\colon d\in D\}$ is
non-empty and compact with respect to the supremum norm. A slightly
weaker condition is assumed in Theorem~\ref{thm:existence}, in the
appendix of this paper.

Unfortunately, it may happen that our beliefs about $X$ cannot be
modelled by a probability measure, simply because we have insufficient
information to identify the probability $\mu(A)$ of every
event $A$ in $\mathcal{F}$. In such a situation, maximising expected
utility usually fails to give an adequate representation of optimality.

For example, let $X$ be the unknown outcome of the tossing of a coin;
say we \emph{only} know that the outcome will be either heads or tails
(so $\mathcal{X}=\{H,T\}$), and that the probability of heads lays
between $28\%$ and $70\%$. Consider the decision set $D=\{1,2,3,4,5,6\}$
and the gain gambles
\begin{align*}
  J_1(H)&=4, & J_1(T)&=0, \\
  J_2(H)&=0, & J_2(T)&=4, \\
  J_3(H)&=3, & J_3(T)&=2, \\
  J_4(H)&=\tfrac{1}{2}, & J_4(T)&=3, \\
  J_5(H)&=\tfrac{47}{20}, & J_5(T)&=\tfrac{47}{20}, \\
  J_6(H)&=\tfrac{41}{10}, & J_6(T)&=-\tfrac{3}{10}, \\
\end{align*}
Clearly, $\opts{\ge}{D}=\{1,2,3,4,5,6\}$, and
\begin{equation*}
\opts{\LinNatXt{\mu}}{D}
=
\begin{cases}
\{2\},   &\text{if }\mu(H)<\frac{2}{5}, \\
\{2,3\}, &\text{if }\mu(H)=\frac{2}{5}, \\
\{3\},   &\text{if }\frac{2}{5}<\mu(H)<\frac{2}{3}, \\
\{1,3\}, &\text{if }\mu(H)=\frac{2}{3}, \\
\{1\},   &\text{if }\mu(H)>\frac{2}{3}.
\end{cases}
\end{equation*}
Concluding, if we have \emph{no} additional information about $X$, but
still insist on using a particular (and necessarily arbitrary) $\mu$,
which is only required to satisfy $0.28\le
\mu(H)\le 0.7$, we find that $\opts{\LinNatXt{\mu}}{D}$ is not
very robust against changes in $\mu$.
This 
shows that maximising expected utility fails to give an adequate
representation of optimality in case of ignorance about the precise
value of $\mu$.

\section{Generalising to Imprecise Probabilities}
\label{sec:generalising:to:improb}

Of course, if we have sufficient information such that
$\mu$ can be identified, nothing is wrong with
Eq.~\eqref{eq:optimal:meu}. We shall therefore try to generalise
Eq.~\eqref{eq:optimal:meu}. In doing so, following Walley \cite{1991:walley}, we shall assume that our
beliefs about $X$ are modelled by a real-valued mapping $\LowPrev{P}$
defined on a---possibly only very small---set $\mathcal{K}$ of
gambles, that represents our assessment of the \emph{lower expected
  utility} $\LowPrev{P}(f)$ for each gamble $f$ in
$\mathcal{K}$;\footnote{The upper expected utility of a gamble $f$ is
  $\UppPrev{P}(f)$ if and only if the lower expected utility of $-f$
  is $-\UppPrev{P}(f)$. So, for any gamble $f$ in $\mathcal{K}$,
  $\LowPrev{P}(-f)=-\UppPrev{P}(f)$, and therefore, without loss of
  generality, we can restrict ourselves to lower expected utility.}
note that $\mathcal{K}$ can be chosen empty if we are completely
ignorant. Essentially, this means that instead of a single probability
measure on $\mathcal{F}$, we now identify a closed convex
set $\mathcal{M}$ of finitely additive probability measures $\mu$ on
$\mathcal{F}$, described by the linear inequalities
\begin{equation}
  \label{eq:credalset}
  (\forall f\in\mathcal{K})(\LowPrev{P}(f)\le\LinNatXt{\mu}(f)).
\end{equation}
We choose the domain $\mathcal{F}$ of the measures $\mu$ sufficiently
large such that all gambles of interest, in particular those in
$\mathcal{K}$ and the gain gambles $J_d$, are measurable with respect
to $\mathcal{F}$. Without loss of generality,
we can assume $\mathcal{F}$ to be the power set of $\mathcal{X}$,
although in practice, it may be more convenient to choose a smaller
field.

For a given $\mathcal{F}$-measurable gamble $g$, not necessarily in $\mathcal{K}$, we may also
derive a lower expected utility $\LowNatXt{\LowPrev{P}}(g)$ by
minimising $\LinNatXt{\mu}(g)$ subject to the above constraints, and
an upper expected utility
$\UppNatXt{\LowPrev{P}}(g)=-\LowNatXt{\LowPrev{P}}(-g)$ by maximising
$\LinNatXt{\mu}(g)$ over the above constraints. In case $\mathcal{X}$ and
$\mathcal{K}$ are finite, this simply amounts to solving a linear
program.

In the literature, $\mathcal{M}$ is called a
\emph{credal set} (see for instance Giron and Rios
\cite{1980:giron:rios::quasi:bayesians}, and Levi \cite[Section~4.2,
pp.~76--78]{1983:levi}, for more comments on this model), and
$\LowPrev{P}$ is called
a \emph{lower prevision} (because they generalise the previsions,
which are fair prices, of De Finetti \cite[Vol.~I, Section~3.1,
pp.~69--75]{1974:definetti}).

The mapping $\LowNatXt{\LowPrev{P}}$ obtained, corresponds exactly to
the so-called \emph{natural extension} of $\LowPrev{P}$ (to the set of
$\mathcal{F}$-measurable gambles), where $\LowPrev{P}(f)$ is
interpreted as a supremum buying price for $f$ (see Walley
\cite[Section~3.4.1, p.~136]{1991:walley}). In this interpretation,
for any $s<\LowPrev{P}(f)$, we are willing to pay any utility
$s<\LowPrev{P}(f)$ prior to observation of $X$, if we are guaranteed
to receive $f(x)$ once $x$ turns out to be the outcome of $X$. The
natural extension then corresponds to the highest price we can obtain
for an arbitrary gamble $g$, taken into account the assessed prices
$\LowPrev{P}(f)$ for $f\in\mathcal{K}$. Specifically,
\begin{equation}
  \label{eq:natxt}
  \LowNatXt{\LowPrev{P}}(g)=
  \sup\left\{
  \alpha+\sum_{i=1}^n\lambda_i\LowPrev{P}(f_i)
  \colon
  \alpha+\sum_{i=1}^n\lambda_i f_i\le g\right\},
\end{equation}
where $\alpha$ varies over $\SetR$, $n$ over $\SetN$, $\lambda_1$,
\dots, $\lambda_n$ vary over $\SetR^+$, and $f_1$, \dots, $f_n$ over
$\mathcal{K}$.

It may happen that $\mathcal{M}$ is empty, in which case
$\LowNatXt{\LowPrev{P}}$ is undefined (the supremum in
Eq.~\eqref{eq:natxt} will always be $+\infty$). This occurs exactly
when $\LowPrev{P}$ \emph{incurs a sure loss} as a lower prevision,
that is, if we can find a finite collection of gambles $f_1$, \dots,
$f_n$ in $\mathcal{K}$ such that
$\sum_{i=1}^n\LowPrev{P}(f_i)>\sup\left[\sum_{i=1}^n f_i\right]$,
which means that we are willing to pay more for this collection than
we can ever gain from it, which makes no sense of course.

Finally, it may happen that $\LowNatXt{\LowPrev{P}}$ does not coincide
with $\LowPrev{P}$ on $\mathcal{K}$. This points to a form of
\emph{incoherence} in $\LowPrev{P}$: this situation occurs exactly
when we can find a finite collection of gambles $f_0$, $f_1$, \dots,
$f_n$ and non-negative real numbers $\lambda_1$, \dots, $\lambda_n$,
such that 
\[
\alpha+\sum_{i=1}^n\lambda_i f_i\le f_0,
\text{ but also }
\LowPrev{P}(f_0)
<
\alpha+\sum_{i=1}^n\lambda_i\LowPrev{P}(f_i).
\]
This means that we can construct a price for $f_0$, using the assessed
prices $\LowPrev{P}(f_i)$ for $f_i$, which is strictly higher than
$\LowPrev{P}(f_0)$. In this sense, $\LowNatXt{\LowPrev{P}}$ corrects
$\LowPrev{P}$, as is apparent from Eq.~\eqref{eq:natxt}.

Although the belief model described above is not the most general we
may think of, it is sufficiently general to model both expected
utility and complete ignorance: these two extremes are obtained by
taking $\mathcal{M}$ either equal to a singleton, or equal to the set
of all finitely additive probability measures on $\mathcal{F}$ (i.e.,
$\mathcal{K}=\emptyset$). It also allows us to demonstrate the
differences between different ways to make decisions with
imprecise probabilities on the example we presented before.

In that example, the given information can be modelled by, say, a
lower prevision $\LowPrev{P}$ on $\mathcal{K}=\{I_H,-I_H\}$, defined
by $\LowPrev{P}(I_H)=0.28$ and
$\LowPrev{P}(-I_H)=-0.7$, where $I_H$ is the gamble defined
by $I_H(H)=1$ and $I_H(T)=0$. For this $\LowPrev{P}$, the set
$\mathcal{M}$ corresponds exactly to the set of all probability measures $\mu$
on $\mathcal{F}=\{\emptyset,\{H\},\{T\},\{H,T\}\}$, such that $0.28\le \mu(H)\le 0.7$. We also easily
find for any gamble $f$ on $X$ that
\begin{equation*}
  \LowNatXt{\LowPrev{P}}(f)=
  \min\{0.28 f(H) + 0.72 f(T), 0.7 f(H) + 0.3 f(T)\}.
\end{equation*}

\subsection{$\Gamma$-Maximin and $\Gamma$-Maximax}

As a very simple way to generalise Eq.~\eqref{eq:optimal:meu}, we
could take the lower expected utility $\LowNatXt{\LowPrev{P}}$ as a
replacement for the expected utility $\LinNatXt{\mu}$ (see for instance
Gilboa and Schmeidler \cite{1989:gilboa::maximin}, or Berger
\cite[Section~4.7.6, pp.~215--223]{1985:berger}):
\begin{equation}
  \opts{\LowNatXt{\LowPrev{P}}}{D}
  :=
  \arg\max_{d\in\opts{\ge}{D}}\LowNatXt{\LowPrev{P}}(J_d);
\end{equation}
this criterion is called \emph{$\Gamma$-maximin}, and amounts to
worst-case optimisation: we take a decision that maximises the worst
expected gain. For example, if we consider the decision as a game
against nature, who is assumed to choose a distribution in
$\mathcal{M}$ aimed at minimizing our expected gain, then the
$\Gamma$-maximin solution is the best we can do. Applied on the
example of Section~\ref{sec:meu:questionmark}, we find as a solution
$\opts{\LowNatXt{\LowPrev{P}}}{D}=\{5\}$.

In case $\mathcal{K}=\emptyset$, i.e., in case of complete ignorance about $X$,
it holds that $\LowNatXt{\LowPrev{P}}(f)=\inf_{x\in\mathcal{X}}f(x)$.
Hence, in that case, $\Gamma$-maximin coincides with maximin (see
Berger \cite[Eq.~(4.96), p.~216]{1985:berger}), ranking
decisions by the minimal (or infimum, to be more precise) value of
their gain gambles.

Some authors consider best-case optimisation, taking a decision that
maximises the best expected gain (see for instance Satia and Lave
\cite{1973:satia}). In our example, the ``$\Gamma$-maximax'' solution is
$\opts{\UppNatXt{\LowPrev{P}}}{D}=\{2\}$.

\subsection{Maximality}

Eq.~\eqref{eq:optimal:meu} is essentially the result of pair-wise
preferences based on expected utility: defining the strict partial
order $\Psg{\mu}$ on $D$ as $d\Psg{\mu}e$ whenever
$\LinNatXt{\mu}(J_d)>\LinNatXt{\mu}(J_e)$, or equivalently, whenever
$\LinNatXt{\mu}(J_d-J_e)>0$, we can simply write
\begin{equation*}
  \opts{\LinNatXt{\mu}}{D}=\maxs{\Psg{\mu}}{\opts{\ge}{D}},
\end{equation*}
where the operator $\maxs{\Psg{\mu}}{\cdot}$ selects the
$\Psg{\mu}$-maximal, i.e., the $\Psg{\mu}$-undominated
elements from a set with strict partial order $\Psg{\mu}$.

Using the supremum buying price interpretation, it is easy to derive
pair-wise preferences from $\LowPrev{P}$: define $\Psg{\LowPrev{P}}$
as $d\Psg{\LowPrev{P}}e$ whenever $\LowNatXt{\LowPrev{P}}(J_d-J_e)>0$.
Indeed, $\LowNatXt{\LowPrev{P}}(J_d-J_e)>0$ means that we are disposed
to pay a strictly positive price in order to take decision $d$ instead
of $e$, which clearly indicates strict preference of $d$ over $e$ (see
Walley \cite[Sections~3.9.1--3.9.3, pp.~160--162]{1991:walley}). Since
$\Psg{\LowPrev{P}}$ is a strict partial order, we arrive at
\begin{align}
  \nonumber \opts{\Psg{\LowPrev{P}}}{D}:=&\maxs{\Psg{\LowPrev{P}}}{\opts{\ge}{D}}
  \\ =&
  \label{eq:optimal:maximality}
  \{d\in\opts{\ge}{D}\colon
  (\forall e\in\opts{\ge}{D})(\LowNatXt{\LowPrev{P}}(J_e-J_d)\le 0)\}
\end{align}
as another generalisation of Eq.~\eqref{eq:optimal:meu}, called
\emph{maximality}. Note that $\Psg{\LowPrev{P}}$ can also be viewed as
a robustification of $\Psg{\mu}$ over $\mu$ in $\mathcal{M}$.  Applied
on the example of Section~\ref{sec:meu:questionmark}, we find
$\opts{\Psg{\LowPrev{P}}}{D}=\{1,2,3,5\}$ as a solution.

Note that Walley \cite[Sections~3.9.2, p.~161]{1991:walley} has a slightly different definition: instead of working from the set of admissible
decisions as in Eq.~\eqref{eq:optimal:maximality}, Walley starts with ranking $d>e$ if
$\LowNatXt{\LowPrev{P}}(J_d-J_e)>0$ or ($J_d\ge J_e$ and $J_d\neq
J_e$), and then selects those decisions from $D$ that are maximal with
respect to this strict partial order. Using
Theorem~\ref{thm:existence:admissible}
from the appendix, it is easy to show that Walley's definition of
maximality coincides with the one given in
Eq.~\eqref{eq:optimal:maximality} whenever the set $\{J_d\colon d\in
D\}$ is weakly compact. This is something we usually assume to ensure the existence of admissible elements; in particular, weak compactness is assumed in Theorem~\ref{thm:existence} (see appendix).
The benefit of Eq.~\eqref{eq:optimal:maximality} over Walley's
definition is that Eq.~\eqref{eq:optimal:maximality} is easier to manage in the proofs in the appendix.

\subsection{Interval Dominance}

Another robustification of $\Psg{\mu}$ is the strict partial ordering
$\Pintvalsg{\LowPrev{P}}$ defined by $d\Pintvalsg{\LowPrev{P}}e$
whenever $\LowNatXt{\LowPrev{P}}(J_d)>\UppNatXt{\LowPrev{P}}(J_e)$;
this means that the interval $[\LowNatXt{\LowPrev{P}}(J_d),\UppNatXt{\LowPrev{P}}(J_d)]$ is completely on the right hand side of the interval $[\LowNatXt{\LowPrev{P}}(J_e),\UppNatXt{\LowPrev{P}}(J_e)]$.
The above ordering is therefore called \emph{interval dominance} (see Zaffalon,
Wesnes, and Petrini \cite[Section~2.3.3,
pp.~68--69]{2003:zaffalon::dementia} for a brief discussion and
references).
\begin{align}
  \nonumber \opts{\Pintvalsg{\LowPrev{P}}}{D}:=&\maxs{\Pintvalsg{\LowPrev{P}}}{\opts{\ge}{D}}
  \\ =&
  \label{eq:optimal:intval}
  \{d\in\opts{\ge}{D}\colon
  (\forall e\in\opts{\ge}{D})(\LowNatXt{\LowPrev{P}}(J_e)\le\UppNatXt{\LowPrev{P}}(J_d))\}
\end{align}
The resulting notion
is weaker than maximality: applied on the example of
Section~\ref{sec:meu:questionmark},
$\opts{\Pintvalsg{\LowPrev{P}}}{D}=\{1,2,3,5,6\}$, which is strictly
larger than $\opts{\Psg{\LowPrev{P}}}{D}$.

\subsection{E-Admissibility}

In the example of Section~\ref{sec:meu:questionmark}, we have shown
that $\opts{\LinNatXt{\mu}}{D}$ may not be very robust against changes
in $\mu$. Robustifying $\opts{\LinNatXt{\mu}}{D}$ against
changes of $\mu$ in $\mathcal{M}$, we arrive at
\begin{equation}\label{eq:eadmissible}
  \opts{\mathcal{M}}{D}
  :=
  \bigcup_{\mu\in\mathcal{M}}
  \opts{\LinNatXt{\mu}}{D};
\end{equation}
this provides another way to generalise Eq.~\eqref{eq:optimal:meu}.
The above criterion selects those admissible decisions in $D$ that
maximize expected utility with respect to at least one $\mu$ in
$\mathcal{M}$; i.e., they select the \emph{E-admissible} (see Good
\cite[p.~114, ll.~8--9]{1952:good}, or Levi \cite[Section~4.8, p.~96,
ll.~8--20]{1983:levi}) decisions among the admissible ones. We find
$\opts{\mathcal{M}}{D}=\{1,2,3\}$ for the example.

In case $\mu$ is defined on $\wp(\mathcal{X})$ and $\mu(\{x\})>0$ for
all $x\in\mathcal{X}$, then every E-admissible decision is also
admissible, and hence, in that case, $\opts{\mathcal{M}}{D}$ gives us
exactly the set of E-admissible options.

\section{Which Is the Right One?}
\label{sec:the:right:one}

Evidently, it is hard to pinpoint the right choice. Instead, let us ask ourselves: what properties do we want
our notion of optimality to satisfy? Let us summarise a few important
guidelines.

Clearly, whatever notion of optimality, it seems reasonable to exclude
inadmissible decisions. For ease of exposition, let's assume that the
inadmissible decisions have already been removed from $D$, i.e.,
$D=\opts{\ge}{D}$; this implies in particular that
$\opts{\mathcal{M}}{D}$ gives us the set of E-admissible decisions.

Now note that, in general, the following implications
hold:
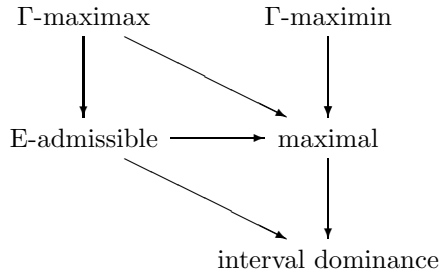
\begin{equation*}
\begin{diagram}
\node{\text{$\Gamma$-maximax}}\arrow{s}\arrow{se} \node{\text{$\Gamma$-maximin}}\arrow{s}\arrow{s} \\
\node{\text{E-admissible}}\arrow{r}\arrow{se} \node{\text{maximal}}\arrow{s}\\
\node{} \node{\text{interval dominance}}
\end{diagram}
\end{equation*}
as is also demonstrated by our example. A proof is given in the
appendix, Theorem~\ref{thm:optimality:connections}.

E-admissibility, maximality, and interval dominance have the nice
property that the more determinate our beliefs (i.e., the smaller
$\mathcal{M}$), the smaller the set of optimal decisions.  In
contradistinction, $\Gamma$-maximin and $\Gamma$-maximax lack this
property, and usually only select a single decision, even in case of
complete ignorance. However, if we are \emph{only} interested in the
most pessimistic (or most optimistic) solution, disregarding other
reasonable solutions, then $\Gamma$-maximin (or $\Gamma$-maximax)
seems appropriate. Utkin and Augustin
\cite{2005:utkin:augustin::algorithms} have collected a number of nice
algorithms for finding $\Gamma$-maximin and $\Gamma$-maximax
solutions, and even mixtures of these two. Seidenfeld
\cite{2004:seidenfeld::maximin:eadm} has compared $\Gamma$-maximin to
E-admissibility, and argued against $\Gamma$-maximin in sequential
decision problems.

If we do not settle for $\Gamma$-maximin (or $\Gamma$-maximax), should
we choose E-admissibility, maximality, or interval dominance?  As
already mentioned, interval dominance is weaker than maximality, so in
general we will end up with a larger (and arguably too large) set of
optimal options. Assuming the non-admissible decisions have been
weeded, a decision $d$ is not optimal in $D$ with respect to interval
dominance if and only if
\begin{equation}\label{eq:optimal:elimination}
  \UppNatXt{\LowPrev{P}}(J_d)<\sup_{e\in D}\LowNatXt{\LowPrev{P}}(J_e).
\end{equation}
Thus, if $D$ has $n$ elements, interval dominance requires us to
calculate $2n$ natural extensions, and make $2n$ comparisons, whereas
for maximality, by Eq.~\eqref{eq:optimal:maximality}, we must
calculate $n^2-n$ natural extensions, and perform $n^2-n$
comparisons---roughly speaking, each natural extension is a linear
program in $m$ (size of $\mathcal{X}$) variables and $r$ (size of
$\mathcal{K}$) constraints, or vice versa if we solve the dual program.
So, comparing maximality and interval
dominance, we face a tradeoff between computational speed and number
of optimal options.

However, this also means that interval dominance is a means
to speed up the calculation of maximal and E-admissible decisions:
because every maximal decision is also interval dominant, we can invoke
interval dominance as a first computationally efficient step in
eliminating non-optimal decisions, if we eventually opt for maximality
or E-admissibility. Indeed, eliminating those decisions $d$ that
satisfy Eq.~\eqref{eq:optimal:elimination}, we will also eliminate
those decisions that are neither maximal, nor E-admissible.

Regarding sequential decision problems, we note that dynamic
programming techniques cannot be used when using interval dominance
(see De Cooman and Troffaes
\cite{decooman:troffaes::dynprog:impgain}), and therefore, since
dynamic programming yields an exponential speedup, maximality and E-admissibility are
certainly preferred over interval dominance once dynamics enter the
picture.

This leaves E-admissibility and maximality. They are quite similar:
they coincide on all decision sets $D$ that contain two decisions. In
case we consider larger decision sets, they coincide if the set of
gain gambles is convex (for instance, if we consider \emph{randomised}
decisions). As already mentioned, E-admissibility is stronger than
maximality, and also has some other advantages over maximality. For
instance, $\frac{1}{5}J_2+\frac{4}{5}J_3\Psg{\LowPrev{P}}J_5$, so,
choosing decision $2$ with probability $20\%$ and decision $3$ with
probability $80\%$ is preferred to decision $5$. Therefore, we should
perhaps not consider decision $5$ as optimal.

E-admissibility is not
vulnerable to such argument, since no E-admissible decision can be
dominated by randomized decisions: if for some $\mu\in\mathcal{M}$ it holds that $\LinNatXt{\mu}(J_d-J_e)\ge 0$
for all $e\in D$, then also 
$$\LinNatXt{\mu}\left(J_d-\sum_{i=1}^n\lambda_i J_{e_i}\right)=\sum_{i=1}^n\lambda_i \LinNatXt{\mu}\left(J_d-J_{e_i}\right)\ge 0$$ 
for any convex combination $\sum_{i=1}^n\lambda_i J_{e_i}$ of gain gambles, and hence, it also holds that 
$$\LowNatXt{\LowPrev{P}}\left(\sum_{i=1}^n\lambda_i J_{e_i}-J_d\right)\le 0$$ 
which means that no convex combination $\sum_{i=1}^n\lambda_i J_{e_i}$ can dominate $J_d$ with respect to $\Psg{\LowPrev{P}}$.

A powerful algorithm for calculating E-admissible options has been
recently suggested by Utkin and Augustin
\cite[pp.~356--357]{2005:utkin:augustin::algorithms}, and independently by Kikuti, Cozman, and de Campos \cite[Sec.~3.4]{2005:cozman:decisiontrees}. If $D$ has $n$
elements, finding all (pure) E-admissible options requires us to solve
$n$ linear programs in $m$ variables and $r+n$ constraints.

As we already noted, through convexification of the decision set,
maximality and E-admissibility coincide. Utkin and Augustin's
algorithm can also cope with this case, but now one has to consider in
the worst case $n!$ linear programs, and usually several less: the worst case
only obtains if all options are E-admissible. For instance, if there
are only $\ell$ E-admissible pure options, one has to consider only at
most $\ell!+n-\ell$ of those linear programs, and again, usually less.

In conclusion, the decision criterion to settle for in a particular
application, depends at least on the goals of the decision maker (what
properties should optimality satisfy?), and possibly also on the size
and structure of the problem if computational issues arise.

\subsection*{Acknowledgements}

I especially want to thank Teddy Seidenfeld for the many instructive
discussions about maximality versus E-admissibility. I also wish to thank two anonymous referees for their helpful comments. This paper has
been supported by the Belgian American Educational Foundation. The
scientific responsibility rests with its author.

\appendix

\section{Proofs}

This appendix is dedicated to proving the connections between the
various optimality criteria, and existence results mentioned
throughout the paper. In the whole appendix, we assume the
following:

Recall, $D$ denotes some set of decisions, and every decision $d\in D$
induces a gain gamble $J_d\in\gambles$, where $\gambles$ is the set of
all gambles (bounded $\mathcal{X}$--$\SetR$ mappings).

$\LowPrev{P}$ denotes a lower prevision, defined on a subset
$\mathcal{K}$ of $\gambles$. With $\mathcal{F}$ we denote a field on
$\mathcal{X}$ such that all gain gambles $J_d$ and gambles in
$\mathcal{K}$ are measurable with respect to $\mathcal{F}$, i.e., are
a uniform limit of $\mathcal{F}$-simple gambles. $\mathcal{F}$ could
be for instance the power set of $\mathcal{X}$.

$\LowPrev{P}$ is assumed to avoid sure
loss, and $\LowNatXt{\LowPrev{P}}$ is its natural extension
to the set of all $\mathcal{F}$-measurable gambles.
$\mathcal{M}$ is the credal set representing $\LowPrev{P}$, as defined
in Section~\ref{sec:generalising:to:improb}. We will make
deliberate use of the properties of natural extension (for instance,
superadditivity:
$\LowNatXt{\LowPrev{P}}(f+g)\ge\LowNatXt{\LowPrev{P}}(f)+\LowNatXt{\LowPrev{P}}(g)$, and hence also $\LowNatXt{\LowPrev{P}}(f-g)\le\LowNatXt{\LowPrev{P}}(f)-\LowNatXt{\LowPrev{P}}(g)$).
We refer to Walley \cite[Sec.~2.6, p.~76, and Sec.~3.1.2,
p.~123]{1991:walley} for an overview and proof of these properties.

We use the symbol $\mu$ for an arbitrary finitely additive probability
measure on $\mathcal{F}$, and $\LinNatXt{\mu}$ denotes the Dunford
integral with respect to $\mu$. This integral is defined on (at least)
the set of all $\mathcal{F}$-measurable gambles.

\subsection{Connections between Decision Criteria}

\begin{thm}\label{thm:optimality:connections}
  The following relations hold.
  \begin{align*}
    &\opts{\UppNatXt{\LowPrev{P}}}{D}\subseteq\opts{\mathcal{M}}{D}\subseteq\opts{\Psg{\LowPrev{P}}}{D}\subseteq\opts{\Pintvalsg{\LowPrev{P}}}{D}\\
    &\opts{\LowNatXt{\LowPrev{P}}}{D}\subseteq\opts{\Psg{\LowPrev{P}}}{D}
  \end{align*}
\end{thm}
\begin{proof}
  Let $\mathcal{J}=\{J_d\colon d\in D\}$.

  Suppose that $d$ is $\Gamma$-maximax in $D$: $J_d$ maximises
  $\UppNatXt{\LowPrev{P}}$ in $\max_{\ge}(\mathcal{J})$. Since
  $\UppNatXt{\LowPrev{P}}$ is the upper envelope of $\mathcal{M}$, and
  $\mathcal{M}$ is weak-* compact (see Walley
  \cite[Sec.~3.6]{1991:walley}), there is a $\mu$ in $\mathcal{M}$
  such that $\UppNatXt{\LowPrev{P}}(J_d)=\LinNatXt{\mu}(J_d)$. But,
  $\LinNatXt{\mu}(J_e)\le\UppNatXt{\LowPrev{P}}(J_e)\le\UppNatXt{\LowPrev{P}}(J_d)=\LinNatXt{\mu}(J_d)$,
  for every $J_e\in\max_{\ge}(\mathcal{J})$ because $d$ is
  $\Gamma$-maximax. Thus, $d$ belongs to $\opts{\mathcal{M}}{D}$.

  Suppose that $d\in\opts{\mathcal{M}}{D}$: there is a $\mu$ in
  $\mathcal{M}$ such that $J_d$ maximises $\LinNatXt{\mu}$ in
  $\max_{\ge}(\mathcal{J})$. But then, because
  $\LowNatXt{\LowPrev{P}}$ is the lower envelope of $\mathcal{M}$,
  $\LowNatXt{\LowPrev{P}}(J_e-J_d)\le\LinNatXt{\mu}(J_e-J_d)=\LinNatXt{\mu}(J_e)-\LinNatXt{\mu}(J_d)\le
  0$ for all $J_e$ in $\max_{\ge}(\mathcal{J})$. Hence, by
  Eq.~\eqref{eq:optimal:maximality} on
  p.~\pageref{eq:optimal:maximality}, $d$ must be maximal.

  Suppose that $d$ is maximal. Then, again by
  Eq.~\eqref{eq:optimal:maximality},
  $\LowNatXt{\LowPrev{P}}(J_e-J_d)\le 0$ for all $J_e$ in
  $\max_{\ge}(\mathcal{J})$. But,
  $\LowNatXt{\LowPrev{P}}(J_e)-\UppNatXt{\LowPrev{P}}(J_d)\le\LowNatXt{\LowPrev{P}}(J_e-J_d)$,
  hence, also
  $\LowNatXt{\LowPrev{P}}(J_e)\le\UppNatXt{\LowPrev{P}}(J_d)$ for all
  $J_e$ in $\max_{\ge}(\mathcal{J})$, which means that $d$ belongs
  to $\opts{\Pintvalsg{\LowPrev{P}}}{D}$.

  Finally, suppose that $d$ is $\Gamma$-maximin: $J_d$ maximises
  $\LowNatXt{\LowPrev{P}}$ in $\max_{\ge}(\mathcal{J})$. But then
  $\LowNatXt{\LowPrev{P}}(J_e-J_d)\le\LowNatXt{\LowPrev{P}}(J_e)-\LowNatXt{\LowPrev{P}}(J_d)\le
  0$ for all $J_e$ in $\max_{\ge}(\mathcal{J})$; $d$ must be maximal.
\end{proof}

\subsection{Existence}

We first prove a technical but very useful lemma about the existence
of optimal elements with respect to preorders; it's an abstraction of
a result proved by De Cooman and Troffaes
\cite{decooman:troffaes::dynprog:impgain}. Let's start with a few definitions.

A \emph{preorder} is simply a reflexive and transitive relation.

Let $\mathcal{V}$ be any set, and let $\trianglerighteqslant{}$ be any
preorder on $\mathcal{V}$. An element $v$ of a subset $\mathcal{S}$ of
$\mathcal{V}$ is called \emph{$\trianglerighteqslant{}$-maximal} in
$\mathcal{S}$ if, for all $w$ in $\mathcal{S}$,
$w\trianglerighteqslant{} v$ implies $v\trianglerighteqslant{} w$. The
set of $\trianglerighteqslant{}$-maximal elements is denoted by
\begin{equation}
  \label{eq:def:maximal:preorder}
  \maxs{\trianglerighteqslant{}}{\mathcal{S}}
  :=
  \Big\{
    v\in\mathcal{S}\colon
    (\forall w\in\mathcal{S})
    (w\trianglerighteqslant{} v\implies v\trianglerighteqslant{} w)
  \Big\}.
\end{equation}
For any $v$ in $\mathcal{S}$, we also define the \emph{up-set} of $v$
relative to $\mathcal{S}$ as
\begin{equation*}
  \UpSet{\trianglerighteqslant{}}{\mathcal{S}}{v}:=\{w\in\mathcal{S}\colon w\trianglerighteqslant{}v\}.
\end{equation*}

\begin{lem}
  \label{lem:preorders:maximality}
  Let $\mathcal{V}$ be a Hausdorff topological space.
  Let $\trianglerighteqslant{}$ be any preorder on $\mathcal{V}$ such that for any
  $v$ in $\mathcal{V}$, the set $\UpSet{\trianglerighteqslant{}}{\mathcal{V}}{v}$ is
  closed. Then, for any non-empty compact subset $\mathcal{S}$ of
  $\mathcal{V}$, the following statements hold.
  \begin{enumerate}[(i)]
  \item\label{lem:preorders:maximality::upset} For every $v$ in
    $\mathcal{S}$, the set $\UpSet{\trianglerighteqslant{}}{\mathcal{S}}{v}$ is
    non-empty and compact.
  \item\label{lem:preorders:maximality::existence} The set
    $\maxs{\trianglerighteqslant{}}{\mathcal{S}}$ of $\trianglerighteqslant{}$-maximal elements of
    $\mathcal{S}$ is non-empty.
  \item\label{lem:preorders:maximality::dominance} For every $v$ in
    $\mathcal{S}$, there is a $\trianglerighteqslant{}$-maximal element $w$ of
    $\mathcal{S}$ such that $w\trianglerighteqslant{}v$.
\end{enumerate}
\end{lem}
\begin{proof}
  \eqref{lem:preorders:maximality::upset}. Since $\trianglerighteqslant{}$ is
  reflexive, it follows that $v\trianglerighteqslant{}v$, so
  $\UpSet{\trianglerighteqslant{}}{\mathcal{S}}{v}$ is non-empty. Is it compact?
  Clearly,
  $\UpSet{\trianglerighteqslant{}}{\mathcal{S}}{v}=\mathcal{S}\cap\UpSet{\trianglerighteqslant{}}{\mathcal{V}}{v}$,
  so $\UpSet{\trianglerighteqslant{}}{\mathcal{S}}{v}$ is the intersection of a
  compact set and a closed set, and therefore
  $\UpSet{\trianglerighteqslant{}}{\mathcal{S}}{v}$ must be compact too.

  \eqref{lem:preorders:maximality::existence}. Let $\mathcal{S}'$ be
  any subset of the non-empty compact set $\mathcal{S}$ that is
  linearly ordered with respect to $\trianglerighteqslant{}$. If we
  can show that $\mathcal{S}'$ has an upper bound in $\mathcal{S}$
  with respect to $\trianglerighteqslant{}$, then we can infer from a
  version of Zorn's lemma \cite[(AC7), p.~144]{1997:schechter} (which
  also holds for preorders) that $\mathcal{S}$ has a
  $\trianglerighteqslant{}$-maximal element. Let then
  $\{v_1,v_2,\dots,v_n\}$ be an arbitrary finite subset of
  $\mathcal{S}'$. We can assume without loss of generality that
  $v_1\trianglerighteqslant{}v_2\trianglerighteqslant{}
  \dots\trianglerighteqslant{}v_n$, and consequently
  $\UpSet{\trianglerighteqslant{}}{\mathcal{S}}{v_1}\subseteq\UpSet{\trianglerighteqslant{}}{\mathcal{S}}{v_2}
  \subseteq\dots\subseteq\UpSet{\trianglerighteqslant{}}{\mathcal{S}}{v_n}$.
  This implies that the intersection
  $\bigcap_{k=1}^n\UpSet{\trianglerighteqslant{}}{\mathcal{S}}{v_k}
  =\UpSet{\trianglerighteqslant{}}{\mathcal{S}}{v_1}$ of these up-sets
  is non-empty: the collection
  $\{\UpSet{\trianglerighteqslant{}}{\mathcal{S}}{v}\colon v\in\mathcal{S}'\}$
  of compact and hence closed ($\mathcal{V}$ is Hausdorff) subsets of
  $\mathcal{S}$ has the finite intersection property. Consequently,
  since $\mathcal{S}$ is compact, the intersection
  $\bigcap_{v\in\mathcal{S}'}\UpSet{\trianglerighteqslant{}}{\mathcal{S}}{v}$
  is non-empty as well, and this is the set of upper bounds of
  $\mathcal{S}'$ in $\mathcal{S}$ with respect to
  $\trianglerighteqslant{}$. So, by Zorn's lemma, $\mathcal{S}$ has a
  $\trianglerighteqslant{}$-maximal element:
  $\maxs{\trianglerighteqslant{}}{\mathcal{S}}$ is non-empty.

  \eqref{lem:preorders:maximality::dominance}. Combine
  \eqref{lem:preorders:maximality::upset} and
  \eqref{lem:preorders:maximality::existence} to show that the
  non-empty compact set $\UpSet{\trianglerighteqslant{}}{\mathcal{S}}{v}$ has a
  maximal element $w$ with respect to $\trianglerighteqslant{}$. It is then a
  trivial step to prove that $w$ is also $\trianglerighteqslant{}$-maximal in
  $\mathcal{S}$: we must show that for any $u$ in $\mathcal{S}$, if
  $u\trianglerighteqslant{}w$, then $w\trianglerighteqslant{}u$. But, if $u\trianglerighteqslant{}w$, then also
  $u\trianglerighteqslant{}v$ since $w\trianglerighteqslant{}v$ by construction. Hence,
  $u\in\UpSet{\trianglerighteqslant{}}{\mathcal{S}}{v}$, and since $w$ is
  $\trianglerighteqslant{}$-maximal in $\UpSet{\trianglerighteqslant{}}{\mathcal{S}}{v}$, it
  follows that $w\trianglerighteqslant{}u$.
\end{proof}

The weak topology on $\gambles$ is
simply the topology of point-wise convergence. That is, a net
$f_\alpha$ in $\gambles$ converges weakly to $f$ in $\gambles$ if
$\lim_\alpha f_\alpha(x)=f(x)$ for all $x\in\mathcal{X}$.

\begin{thm}\label{thm:existence:admissible}
  If $\mathcal{J}=\{J_d\colon d\in D\}$ is a non-empty and weakly
  compact set, then $D$ contains at least one admissible decision, and
  even more, for every decision $e$ in $D$, there is an
  admissible decision $d$ in $D$ such that $J_d\ge J_e$.
\end{thm}
\begin{proof}
  It is easy to derive from Eq.~\eqref{eq:optimal:pointwise} that
  \begin{equation*}
    \opts{\ge}{D}=
    \{d\in D\colon (\forall e\in D)(J_e\ge J_d\implies J_d\ge J_e)\}.
  \end{equation*}
  Hence, a decision is admissible in $D$ exactly when its gain gamble
  is $\ge$-maximal in $\mathcal{J}$. We must show that $\mathcal{J}$
  has $\ge$-maximal elements.

  By Lemma~\ref{lem:preorders:maximality}, it suffices to prove that,
  for every $f\in\gambles$, the set $\mathcal{G}_f=\{g\in\gambles\colon
  g\ge f\}$ is closed with respect to the topology of point-wise
  convergence.

  Let $g_\alpha$ be a net in $\mathcal{G}_f$, and suppose that
  $g_\alpha$ converges point-wise to $g\in\gambles$: for every
  $x\in\mathcal{X}$, $\lim_\alpha g_\alpha(x)=g(x)$.  But, since
  $g_\alpha(x)\ge f(x)$ for every $\mathcal{X}$, it must also hold
  that $g(x)=\lim_\alpha g_\alpha(x)\ge f(x)$. Hence,
  $g\in\mathcal{G}_f$. We have shown that every converging net in
  $\mathcal{G}_f$ converges to a point in $\mathcal{G}_f$. Thus,
  $\mathcal{G}_f$ is closed. This establishes the theorem.
\end{proof}

Let's now introduce a slightly stronger topology on $\gambles$. This
topology has no particular name in the literature, so let's just call
it the $\tau$-topology. It is determined by the following convergence.

\begin{defn}
  Say that a net
  $f_\alpha$ in $\gambles$ \emph{$\tau$-converges} to $f$ in $\gambles$, if
  \begin{enumerate}[(i)]
  \item $\lim_\alpha f_\alpha(x)=f(x)$ for all $x\in\mathcal{X}$ (point-wise convergence), and
  \item $\lim_\alpha\UppNatXt{\LowPrev{P}}(|f_\alpha-f|)=0$ (convergence in $\UppNatXt{\LowPrev{P}}(|\cdot|)$-norm).
  \end{enumerate}
\end{defn}

This convergence induces a topology $\tau$ on $\gambles$: it turns $\gambles$
into a locally convex topological vector space, which also happens to
be Hausdorff. A topological basis at $0$ consists for instance of the
convex sets
\begin{equation*}
  \{f\in\gambles\colon \UppNatXt{\LowPrev{P}}(|f|)<\epsilon\text{ and } f(x)<\delta(x)\},
\end{equation*}
for $\epsilon>0$, and $\delta(x)>0$ for all $x\in\mathcal{X}$. It has
more open sets and more closed sets than the weak topology, but it has
less compact sets than the weak topology. On the other hand, this
topology is weaker than the supremum norm topology, so it has fewer
open and closed sets, and more compact sets, compared to the supremum
norm topology. Note that in case $\mathcal{X}$ is finite, it reduces
to the weak topology, which is in that case also equivalent to the
supremum norm topology.

Note that $\LowNatXt{\LowPrev{P}}$, $\UppNatXt{\LowPrev{P}}$, and
$\LinNatXt{\mu}$ for all $\mu\in\mathcal{M}$, are $\tau$-continuous,
simply because
\begin{align*}
\UppNatXt{\LowPrev{P}}(|f_\alpha-f|)&\ge
|\LowNatXt{\LowPrev{P}}(f_\alpha)-\LowNatXt{\LowPrev{P}}(f)|,\\
\UppNatXt{\LowPrev{P}}(|f_\alpha-f|)&\ge
|\UppNatXt{\LowPrev{P}}(f_\alpha)-\UppNatXt{\LowPrev{P}}(f)|,\text{ and }\\
\UppNatXt{\LowPrev{P}}(|f_\alpha-f|)&\ge
|\LinNatXt{\mu}(f_\alpha)-\LinNatXt{\mu}(f)|
\end{align*}
(see Walley \cite[p.~77, Sec.~2.6.1(l)]{1991:walley}).
We will exploit this fact in the proof of the following theorem, generalising a result due to Walley \cite[p.~161, Sec.~3.9.2]{1991:walley}.

\begin{thm}\label{thm:existence}
  If $\mathcal{J}=\{J_d\colon d\in D\}$ is non-empty and compact with
  respect to the $\tau$-topology, then the following statements hold.
  \begin{enumerate}[(i)]
  \item\label{thm:existence:mu}
    $\opts{\LinNatXt{\mu}}{D}$ is non-empty for all $\mu\in\mathcal{M}$.
  \item\label{thm:existence:maximin}
    $\opts{\LowNatXt{\LowPrev{P}}}{D}$ is non-empty.
  \item\label{thm:existence:maximax}
    $\opts{\UppNatXt{\LowPrev{P}}}{D}$ is non-empty.
  \item\label{thm:existence:maximal}
    $\opts{\Psg{\LowPrev{P}}}{D}$ is non-empty.
  \item\label{thm:existence:intval}
    $\opts{\Pintvalsg{\LowPrev{P}}}{D}$ is non-empty.
  \item\label{thm:existence:eadm}
    $\opts{\mathcal{M}}{D}$ is non-empty.
  \end{enumerate}
\end{thm}
\begin{proof}
  \eqref{thm:existence:mu}. Introduce the following order
  on $\gambles$: say that $f\trianglerighteqslant{}g$ whenever
  $\LinNatXt{\mu}(f)\ge\LinNatXt{\mu}(g)$. Let's first show
  that, for all $f\in\gambles$, the set
  $\mathcal{G}_f=\{g\in\gambles\colon g\trianglerighteqslant{} f\}$ is
  $\tau$-closed.

  Let $g_\alpha$ be a net in $\mathcal{G}_f$, and suppose that
  $g_\alpha$ $\tau$-converges to $g\in\gambles$.
  Since the integral $\LinNatXt{\mu}$
  is $\tau$-continuous, it follows that
  $\LinNatXt{\mu}(g)=\lim_\alpha\LinNatXt{\mu}(g_\alpha)\ge\LinNatXt{\mu}(f)$.
  Concluding, $g$ belongs to $\mathcal{G}_f$. We have established that
  every converging net in $\mathcal{G}_f$ converges to a point in
  $\mathcal{G}_f$. Thus, $\mathcal{G}_f$ is $\tau$-closed.

  By Lemma~\ref{lem:preorders:maximality}, it follows that
  $\mathcal{J}$ has at least one $\trianglerighteqslant{}$-maximal
  element $J_e$, that is, $J_e$ maximises $\LinNatXt{\mu}$ in
  $\mathcal{J}$. Since any $\tau$-compact set is also weakly compact,
  there is a $\ge$-maximal element $J_d$ in $\mathcal{J}$ such that
  $J_d\ge J_e$, by Theorem~\ref{thm:existence:admissible}. But then,
  $\LinNatXt{\mu}(J_d)\ge\LinNatXt{\mu}(J_e)$, and hence, $J_d$ also
  maximises $\LinNatXt{\mu}$ in $\mathcal{J}$.
  Because $J_d$ is $\ge$-maximal in $\mathcal{J}$, it
  also maximises $\LinNatXt{\mu}$ in
  $\max_{\ge}(\mathcal{J})$. This establishes that $d$ belongs to
  $\opts{\LinNatXt{\mu}}{D}$: this set is non-empty.

  \eqref{thm:existence:maximin}. Introduce the following order on
  $\gambles$: say that $f\trianglerighteqslant{}g$ whenever
  $\LowNatXt{\LowPrev{P}}(f)\ge\LowNatXt{\LowPrev{P}}(g)$. Continue along the
  lines of \eqref{thm:existence:mu}, using the fact that
  $\LowNatXt{\LowPrev{P}}$ is $\tau$-continuous.

  \eqref{thm:existence:maximax}. Again along the lines of
  \eqref{thm:existence:mu}, with
  $f\trianglerighteqslant{}g$ whenever 
  $\UppNatXt{\LowPrev{P}}(f)\ge\UppNatXt{\LowPrev{P}}(g)$.

  \eqref{thm:existence:maximal}\&\eqref{thm:existence:intval}\&\eqref{thm:existence:eadm}.
  Immediate, by \eqref{thm:existence:maximax} and
  Theorem~\ref{thm:optimality:connections}.
\end{proof}


\begin{thebibliography}{10}
\expandafter\ifx\csname url\endcsname\relax
  \def\url#1{\texttt{#1}}\fi
\expandafter\ifx\csname urlprefix\endcsname\relax\def\urlprefix{URL }\fi

\bibitem{1944:neumann}
J.~von Neumann, O.~Morgenstern, Theory of Games and Economic Behavior,
  Princeton University Press, 1944.

\bibitem{1953:herstein:milnor}
I.~N. Herstein, J.~Milnor, An axiomatic approach to measurable utility,
  Econometrica 21~(2) (1953) 291--297.

\bibitem{1985:berger}
J.~O. Berger, Statistical Decision Theory and Bayesian Analysis, 2nd Edition,
  Springer, 1985.

\bibitem{1961:raiffa:schlaifer}
H.~Raiffa, R.~Schlaifer, Applied Statistical Decision Theory, MIT Press, 1961.

\bibitem{1983:levi}
I.~Levi, The Enterprise of Knowledge. An Essay on Knowledge, Credal
  Probability, and Chance, MIT Press, Cambridge, 1983.

\bibitem{1935:dunford}
N.~Dunford, Integration in general analysis, Transactions of the American
  Mathematical Society 37~(3) (1935) 441--453.

\bibitem{1957:dunford:schwartz}
N.~Dunford, J.~T. Schwartz, Linear Operators, John Wiley {\&} Sons, New York,
  1957.

\bibitem{2002:kallenberg}
O.~Kallenberg, Foundations of Modern Probability, 2nd Edition, Probability and
  Its Applications, Springer, 2002.

\bibitem{1991:walley}
P.~Walley, Statistical Reasoning with Imprecise Probabilities, Chapman and
  Hall, London, 1991.

\bibitem{1980:giron:rios::quasi:bayesians}
F.~J. Giron, S.~Rios, Quasi-{B}ayesian behaviour: A more realistic approach to
  decision making?, in: J.~M. Bernardo, J.~H. DeGroot, D.~V. Lindley, A.~F.~M.
  Smith (Eds.), Bayesian Statistics, University Press, Valencia, 1980, pp.
  17--38.

\bibitem{1974:definetti}
B.~De~Finetti, Theory of Probability: A Critical Introductory Treatment, Wiley,
  New York, 1974--5, two volumes.

\bibitem{1989:gilboa::maximin}
I.~Gilboa, D.~Schmeidler, Maxmin expected utility with non-unique prior,
  Journal of Mathematical Economics 18~(2) (1989) 141--153.

\bibitem{1973:satia}
J.~K. Satia, J.~Roy E.~Lave, Markovian decision processes with uncertain
  transition probabilities, Operations Research 21~(3) (1973) 728--740.

\bibitem{2003:zaffalon::dementia}
M.~Zaffalon, K.~Wesnes, O.~Petrini, Reliable diagnoses of dementia by the naive
  credal classifier inferred from incomplete cognitive data, Artificial
  Intelligence in Medicine 29~(1--2) (2003) 61--79.

\bibitem{1952:good}
I.~J. Good, Rational decisions, Journal of the Royal Statistical Society,
  Series B 14~(1) (1952) 107--114.

\bibitem{2005:utkin:augustin::algorithms}
L.~V. Utkin, T.~Augustin, Powerful algorithms for decision making under partial
  prior information and general ambiguity attitudes, in: F.~G. Cozman, R.~Nau,
  T.~Seidenfeld (Eds.), Proceedings of the Fourth International Symposium on
  Imprecise Probabilities and Their Applications, 2005, pp. 349--358.

\bibitem{2004:seidenfeld::maximin:eadm}
T.~Seidenfeld, A contrast between two decision rules for use with (convex) sets
  of probabilities: {G}amma-maximin versus {E}-admissibility, Synthese
  140~(1--2) (2004) 69--88.

\bibitem{decooman:troffaes::dynprog:impgain}
G.~{de Cooman}, M.~C.~M. Troffaes, Dynamic programming for deterministic
  discrete-time systems with uncertain gain, International Journal of
  Approximate Reasoning 39~(2--3) (2004) 257--278.

\bibitem{2005:cozman:decisiontrees}
D.~Kikuti, F.~G. Cozman, C.~P. de~Campos, Partially ordered preferences in
  decision trees: Computing strategies with imprecision in probabilities, in:
  R.~Brafman, U.~Junker (Eds.), Multidisciplinary IJCAI-05 Workshop on Advances
  in Preference Handling, 2005, pp. 118--123.

\bibitem{1997:schechter}
E.~Schechter, Handbook of Analysis and Its Foundations, Academic Press, San
  Diego, 1997.

\end{thebibliography}

\end{document}